\documentclass{article}
\usepackage{geometry}
\usepackage[utf8]{inputenc} 
\usepackage[T1]{fontenc}    
\usepackage{url}            
\usepackage{booktabs}       
\usepackage{amsfonts}       
\usepackage{nicefrac}       
\usepackage{microtype}      
\usepackage{xcolor}         

\usepackage{amsmath}
\usepackage{amssymb}
\usepackage{mathtools}
\usepackage{amsthm}
\usepackage{bm}
\theoremstyle{plain}
\newtheorem{theorem}{Theorem}[section]
\newtheorem{proposition}[theorem]{Proposition}
\newtheorem{lemma}[theorem]{Lemma}
\newtheorem{corollary}[theorem]{Corollary}
\theoremstyle{definition}

\theoremstyle{remark}

\usepackage{algorithm,algorithmic}
\usepackage{subcaption}
\usepackage{graphicx}

\title{Efficient Data-Driven Leverage Score Sampling Algorithm for the Minimum Volume Covering Ellipsoid Problem in Big Data}

\author{
\textbf{Elizabeth Harris}$^{1}$ \quad 
\textbf{Ali Eshragh}$^{2,3}$ \quad
\textbf{Bishnu Lamichhane}$^{1}$ \\
\textbf{Jordan Shaw-Carmody}$^{1}$ \quad
\textbf{Elizabeth Stojanovski}$^{1}$\\
$^1$School of Information and Physical Sciences, University of Newcastle, NSW, Australia \\
$^2$Carey Business School, Johns Hopkins University, MD, USA \\
$^3$International Computer Science Institute, Berkeley, CA, USA
}

\begin{document}

\maketitle

\begin{abstract}
    The Minimum Volume Covering Ellipsoid (MVCE) problem, characterised by $n$ observations in $d$ dimensions where $n \gg d$, can be computationally very expensive in the big data regime. We apply methods from randomised numerical linear algebra to develop a data-driven leverage score sampling algorithm for solving MVCE, and establish theoretical error bounds and a convergence guarantee. Assuming the leverage scores follow a power law decay, we show that the computational complexity of computing the approximation for MVCE is reduced from $\mathcal{O}(nd^2)$ to $\mathcal{O}(nd + \text{poly}(d))$, which is a significant improvement in big data problems. Numerical experiments demonstrate the efficacy of our new algorithm, showing that it substantially reduces computation time and yields near-optimal solutions.
\end{abstract}
\section{Introduction}
\label{sec:intro}
The Minimum Volume Covering Ellipsoid (MVCE) problem arises in many applied and theoretical areas. Statistical applications include outlier detection \cite{titterington1978estimation}, clustering \cite{rosen1965pattern}, and the closely related D-optimal design problem \cite{silvey1980optimal}. In fact, the MVCE problem and the D-optimal design problem are dual to one another \cite{sibson1972discussion,titterington1975optimal}. Containing ellipsoids are used in parameter identification and control theory to describe uncertainty sets for parameters and state vectors \cite{chernousko2005ellipsoidal,schweppe1968recursive}. Minimum volume covering ellipsoids are also used in computational geometry and computer graphics \cite{eberly20063d}, in particular, for collision detection \cite{chen2016virtual}.

Many algorithms for computing MVCEs and D-optimal designs have been studied in the literature. These include Frank-Wolfe type algorithms \cite{frank1956algorithm, wolfe1970convergence, atwood1973sequences, khachiyan1996rounding, kumar2005minimum}, interior point algorithms \cite{khachiyan1993complexity, nesterov1994interior}, the Dual Reduced Newton algorithm \cite{sun2004computation}, the Cocktail algorithm \cite{yu2011d}, the Randomised Exchange algorithm \cite{harman2020randomized}, and the Fixed Point algorithm \cite{cohen2019near,woodruff2023new}. However, when these algorithms are applied to very large datasets, they may be computationally inefficient, and may exceed storage limitations \cite{harman2020randomized}. One solution is to combine such solution algorithms with active set or batching strategies \cite{sun2004computation, kallberg2019active, kudela2019minimum, rosa2022computing}. The main idea of this approach is to iteratively apply the solution algorithm to a smaller subset of points until convergence to the solution. The motivation for this approach is that an ellipsoid is only determined by at most $d(d+3)/2$ points on its boundary \cite{john1948extremum}.

Instead of applying an active set strategy, we apply deterministic sampling to reduce the number of points considered by the algorithm. This deterministic sampling method selects points corresponding to the highest statistical leverage scores. The resulting compressed dataset approximately maintains many of the qualities of the original dataset, provided that the number of samples is sufficiently large \cite{papailiopoulos2014provable, mccurdy2019deterministic}. However, there is no immediate guarantee of the quality of solution to the MVCE problem for this compressed dataset. In this paper, we provide the first theoretical guarantees on the quality of initial and final solution for the MVCE problem for this deterministic sampling method. In numerical experiments, we apply the MVCE algorithm to just our sampled data. Our experiments show that this sampling method greatly improves computation time, and provides near-optimal final solutions.

For the remainder of Section \ref{sec:intro}, we outline the contributions and limitations of our approach. In Sections \ref{sec:MVCEproblem} and \ref{sec:lev}, we formally introduce the MVCE problem, and leverage score based sampling, respectively. Our approach is presented in Section \ref{sec:Approach}, and our theoretical results in Sections \ref{sec:TheoreticalResults} and \ref{sec:TheoreticalResults2}. In Section \ref{sec:empirical_results}, we show our numerical results. We present our conclusions in Section \ref{sec:Conclusion}. Some proofs are provided in Appendices \ref{app:altproof}, \ref{app:dalss}, \ref{app:initoptgap}, \ref{app:finaloptgap}, and additional numerical results are presented in Appendix \ref{app:numresults}.
\subsection{Contributions} \label{sec:contributions}
In this paper, we present a new simplified proof of Theorem 1 in \cite{mccurdy2019deterministic}. Using this theorem, we provide the first theoretical guarantees on the quality of initial and final solution for the MVCE problem when using deterministic leverage score sampling. Further, we show that these guarantees still hold (with high probability) when approximate leverage scores are used. 

Assuming a power law decay on the leverage scores, we show that our method improves the theoretical computation time required to approximate the MVCE. We also demonstrate the efficiency of our approach on synthetic and real world datasets.
\subsection{Limitations} \label{sec:limitations}
We require the leverage scores to follow a power law decay, so that our sample size $s$ is guaranteed to be much smaller than $n$. However, in our numerical experiments, we demonstrate that this requirement can usually be relaxed in practice. In future work, we aim to remove the power law decay requirement.

We also assume $n \gg d$. This is because our base algorithm, the Wolfe-Atwood (WA) algorithm \cite{wolfe1970convergence,atwood1973sequences}, performs best under this assumption. However, if this condition is violated, we can use a different underlying algorithm which performs well when $n$ and $d$ are of similar order (e.g., the Fixed Point algorithm \cite{cohen2019near,woodruff2023new}).

Finally, we make no assumption on the sparsity of our data matrix $\bm{X}$. Although this gives wider applicability, it means we do not take advantage of any sparse input. Again, this can be resolved by using a different underlying algorithm (e.g., the Fixed Point algorithm \cite{cohen2019near,woodruff2023new}).
\paragraph{Notation.}
Throughout this paper, vectors and matrices are denoted by bold lowercase and bold uppercase letters, respectively (e.g. $\bm{a}$ and $\bm{A}$). The $i$th entry of $\bm{a}$ is denoted $a_i$, and the $(i,j)$th entry of $\bm{A}$ is denoted $\bm{A}_{i\,j}$. Let $\bm{A}$ and $\bm{B}$ be symmetric positive definite matrices, then $\bm{A} \succeq \bm{B}$ if $\bm{A} - \bm{B}$ is symmetric positive semidefinite; and $\bm{A} \succ \bm{B}$ if $\bm{A} - \bm{B}$ is symmetric positive definite. The determinant of a matrix $\bm{A}$ is denoted $\det \left( \bm{A} \right)$. Unless otherwise specified, $\bm{A} = \mathtt{diag}\left( \bm{a} \right)$. We use regular lowercase to denote scalar constants (e.g. $c$). Finally, $\bm{e}$ denotes a vector of ones, and $\bm{e}_i$ denotes the vector with one at position $i$, and zero otherwise.
\section{The minimum volume covering ellipsoid problem}
\label{sec:MVCEproblem}
Let $\mathcal{X} = \left\lbrace \bm{x}_1,\dots,\bm{x}_n \right\rbrace$ be a set of data points in $\text{\rm{I\!R}}^d$. Then the minimum volume covering ellipsoid is the ellipsoid that covers $\mathcal{X}$, which attains the minimum volume of all covering ellipsoids of $\mathcal{X}$. We assume throughout that there exists a non-degenerate minimum volume covering ellipsoid.

We define the ellipsoid $\mathcal{E}\left(\bm{Q},\bm{x}_c\right)$ as
\begin{align*}
    \mathcal{E} \left(\bm{Q},\bm{x}_c\right) := \{\bm{x} \in \text{\rm{I\!R}}^d \, : \, \left(\bm{x} - \bm{x}_c\right)^\intercal \bm{Q} \left(\bm{x} - \bm{x}_c\right) \leq d \},
\end{align*}
where $\bm{x}_c$ is the centre of the ellipsoid, and $\bm{Q}$ is an $d$-dimensional symmetric positive definite matrix. Its volume is given by
\begin{align*}
    \text{vol} \left(\mathcal{E}\left(\bm{Q},\bm{x}_c\right)\right) = d^{d/2} \Omega_d \det\left(\bm{Q}\right)^{-1/2},
\end{align*}
where $\Omega_d$ is the volume of the unit ball in $\text{\rm{I\!R}}^d$ (e.g., see \cite{todd2016book}). 

We can now write the mathematical formulation of the Minimum Volume Covering Ellipsoid problem. Suppose we have a finite set of points $\mathcal{X} = \left\lbrace \bm{x}_1, \dots, \bm{x}_n \right\rbrace \subset \text{\rm{I\!R}}^d$. Then its minimum volume covering ellipsoid can be found by solving
\begin{align}
    \underset{\bm{Q} \succ \bm{0},\, \bm{x}_c \in \text{\rm{I\!R}}^d}{\text{minimise}} \quad & -\log \det\left(\bm{Q}\right) \qquad \tag{\textsc{P}$_0$} \label{Primal1}\\
    \text{subject to} \quad & \left(\bm{x}_i - \bm{x}_c\right)^\intercal \bm{Q} \left(\bm{x}_i - \bm{x}_c\right) \leq d, \quad i = 1,\dots,n. \nonumber
\end{align}
Although the objective function is convex, \eqref{Primal1} itself is not convex \cite{todd2016book}. 

We therefore reformulate \eqref{Primal1} so that it is convex. At the cost of working in $\text{\rm{I\!R}}^{d+1}$, we can calculate the centred minimum volume covering ellipsoid, and recover the solution to \eqref{Primal1} \cite{titterington1975optimal}. Therefore, we can set $\bm{x}_c = \bm{0}$, and obtain
\begin{align}
    \underset{\bm{Q} \succ \bm{0}}{\text{minimise}} \quad & f(\bm{Q}) := -\log \det\left(\bm{Q}\right) \qquad \tag{\textsc{P}} \label{Primal} \\
    \text{subject to} \quad & \bm{x}_i^\intercal \bm{Q} \bm{x}_i \leq d, \quad i = 1,\dots,n. \nonumber
\end{align}
 We let $\bm{Q}^*$ denote an optimal solution to Problem \eqref{Primal}. Then MVCE$\left(\mathcal{X} \right) := \mathcal{E}\left(\bm{Q}^*, \bm{0} \right)$ is the minimum volume covering ellipsoid. We will refer to Problem \eqref{Primal} as the MVCE problem.

 The dual problem to Problem \eqref{Primal} is the D-optimal design problem. This problem is concave, and can be formulated as
 \begin{align}
    \underset{\bm{u} \in \rm{I\!R}^n}{\text{maximise}} \quad & g(\bm{u}) := \log \det\left( \sum_{i=1}^n u_i \bm{x}_i \bm{x}_i^\intercal \right) \tag{\textsc{D}} \label{Dual} \\
     \text{subject to} \quad & \sum_{i=1}^n u_i = 1, \quad \bm{u} \geq \bm{0}, \nonumber
 \end{align}
where $\bm{u}$ is called the design vector. We note that for every design vector $\bm{u}$, we can find its associated shape matrix
\begin{align*}
    \bm{Q}\left(\bm{u}\right) := \left( \sum_{i=1}^n u_i \bm{x}_i \bm{x}_i^\intercal \right)^{-1},
\end{align*}
provided that the inverse exists. Hence if we have an optimal solution $\bm{u}^*$ to Problem \eqref{Dual}, then the optimal solution to Problem \eqref{Primal} is $\bm{Q}\left(\bm{u}^*\right)$.
 
We note that Problem \eqref{Dual} (and Problem \eqref{Primal}) cannot usually be solved exactly, so we will focus on deriving approximate solutions. To ensure the chosen algorithm terminates with a guaranteed quality of solution, we will define some approximate optimality conditions. A feasible $\bm{u}$ for Problem \eqref{Dual} is called $\delta$-primal feasible if $\bm{Q}\left(\bm{u}\right)$ satisfies
\begin{align*}
    \bm{x}_i^\intercal \bm{Q}\left(\bm{u}\right) \bm{x}_i \leq \left(1+\delta\right)d,
\end{align*}
for all $i = 1,\dots,n$. If $\bm{Q}\left(\bm{u}\right)$ additionally satisfies
\begin{align*}
    \bm{x}_i^\intercal \bm{Q}\left(\bm{u}\right) \bm{x}_i \geq \left(1-\delta\right)d \text{ if } u_i > 0,
\end{align*}
for all $i = 1,\dots,n$, then we say that $\bm{u}$ is $\delta$-approximately optimal. These optimality conditions ensure the optimality gap is small.
\begin{proposition}[\cite{todd2016book}, Proposition 2.9] \label{prop:optimality gap}
    If we have a $\delta$-primal feasible (or $\delta$-approximately optimal) solution $\bm{u}$, then $\bm{u}$ and $(1+\delta)^{-1} \bm{Q}\left( \bm{u} \right)$ are both within $d \log \left( 1 + \delta \right)$ of being optimal in \eqref{Dual} and \eqref{Primal}, respectively.
\end{proposition}
\section{Sampling using leverage scores}
\label{sec:lev}
The concept of statistical leverage scores has long been used in statistical regression diagnostics to identify outliers \cite{rousseeuw2011robust}. Given a data matrix $\bm{X} \in \text{\rm{I\!R}}^{n \times d}$ with $n > d$, consider any orthogonal matrix $\bm{A}$ such that $\text{Range}(\bm{A}) = \text{Range}(\bm{X})$. The $i$th leverage score corresponding to the $i$th row of $\bm{X}$ is defined as
\begin{align*}
    \ell_i \left( \bm{X} \right) := \left\Vert \bm{A}(i,:)\right\Vert^2. 
\end{align*}
It can be easily shown that this is well defined in that the leverage score does not depend on the particular choice of the basis matrix $\bm{A}$. Furthermore, the $i$th leverage score is the $i$th diagonal entry of the hat matrix, that is,
\begin{align*}
    \ell_i \left( \bm{X} \right) = \bm{e}_i^\intercal \bm{H} \bm{e}_i, \qquad i = 1,\dots,n, 
\end{align*}
where
\begin{align*}
    \bm{H} := \bm{X} (\bm{X}^\intercal \bm{X})^{-1} \bm{X}^\intercal.
\end{align*}
The hat matrix is symmetric and idempotent. We can use these properties to easily show that
\begin{align*}
    0 \leq \ell_i \left( \bm{X} \right) \leq 1,
\end{align*}
for all $i$, and
\begin{align*}
    \sum_{i=1}^n \ell_i \left( \bm{X} \right) = \text{rank}\left( \bm{X} \right).
\end{align*}
We only consider full rank matrices, $\bm{X}$, that is, with $\text{rank}\left( \bm{X} \right) = d$.
Thus
\begin{align*}
    \sum_{i=1}^n \ell_i \left( \bm{X} \right) = d.
\end{align*}
\subsection{Deterministic sampling} \label{subsec:detsamp}
When sampling deterministically, we sample the $s$ rows from $\bm{X}$ with highest leverage scores. Without loss of generality, assume that the leverage scores are ordered $\ell_1 \left( \bm{X} \right) \geq \dots \geq \ell_n \left( \bm{X} \right)$. We summarise the sampling procedure in Algorithm \ref{alg:DLSS_thresh}.

\begin{algorithm}
    \begin{algorithmic}[1]
        \REQUIRE $\bm{X} = \begin{bmatrix} \bm{x}_1, \dots, \bm{x}_n \end{bmatrix}^\intercal \in \text{\rm{I\!R}}^{n \times d},  \varepsilon \in (0,1)$
        \STATE Calculate leverage scores for each row in $\bm{X}$. Without loss of generality, let $\ell_1 \left( \bm{X} \right) \geq \dots \geq \ell_n \left( \bm{X} \right)$.
        \STATE Let $s = \arg \min_j \left( \sum_{i=1}^j \ell_i \left( \bm{X} \right) > d - \varepsilon \right)$.
        \STATE Let $\bm{R} = \bm{0} \in \text{\rm{I\!R}}^{s \times n}$.
        \FOR{$i = 1:s$}
            \STATE Set row $i$ of $\bm{R}$ equal to $\bm{e}_i$.
        \ENDFOR
        \ENSURE $\bm{R}, s$
        \caption{Deterministic Leverage Score Sampling with Threshold \cite{papailiopoulos2014provable}}
        \label{alg:DLSS_thresh}
    \end{algorithmic}
\end{algorithm}

We note that $s$ is carefully chosen, to ensure the following subspace embedding result.
\begin{theorem}[\cite{mccurdy2019deterministic}, Theorem 1] \label{thrm:thresh}
    Let $\varepsilon \in (0,1)$. Use Algorithm \ref{alg:DLSS_thresh} to construct the sampling matrix $\bm{R}$. Then
    \begin{align*}
        (1-\varepsilon) \bm{X}^\intercal \bm{X} \prec \left(\bm{R} \bm{X}\right)^\intercal \bm{R} \bm{X} \preceq \bm{X}^\intercal \bm{X}.
    \end{align*}
\end{theorem}

Algorithm \ref{alg:DLSS_thresh} has time complexity $\mathcal{O} \left(nd^2\right)$, due to the cost of calculating the leverage scores exactly. This can be improved by using approximate leverage score algorithms. Drineas et al. \cite{drineas2012fast} developed a fast sampled randomised Hadamard transform (SRHT) with time complexity $\mathcal{O} \left(nd \log n \right)$. More recently, Eshragh et al. \cite{eshragh2023sequential} developed the Sequential Approximate Leverage-Score Algorithm (SALSA) with time complexity $\mathcal{O} \left( nd \right)$. Moreover, we can use these approximate leverage score algorithms to prove an analogous result to Theorem \ref{thrm:thresh} (see Appendix \ref{app:dalss}).
\section{Our approach}
\label{sec:Approach}
We are interested in calculating MVCE$\left(\mathcal{X} \right)$, where $\mathcal{X} = \left\lbrace \bm{x}_1, \dots, \bm{x}_n \right\rbrace \subset \text{\rm{I\!R}}^d$, $n \gg d$. For very large $n$ (and $d$), this can be very computationally expensive. Instead of calculating MVCE$\left(\mathcal{X} \right)$ directly, we will calculate the MVCE on a much smaller subset of $\mathcal{X}$. We select our subset by using leverage score based sampling, as introduced in Section \ref{sec:lev}. Let this subset be $\mathcal{X}_s$, where $\mathcal{X}_s$ contains up to $s$ points sampled from $\mathcal{X}$. For ease of calculation, let the points from $\mathcal{X}$ be stored in the rows of the data matrix $\bm{X} \in \text{\rm{I\!R}}^{n \times d}$. Then we sample $s$ rows from $\bm{X}$, which we store in the rows of $\bm{X}_s \in \text{\rm{I\!R}}^{s \times d}$. We then use $\bm{X}_s$ instead of $\bm{X}$ in our solution algorithm. That is, our solution algorithm solves the concave problem
\begin{align}
    \underset{\bm{u} \in \rm{I\!R}^s}{\text{maximise}} \quad & g_s(\bm{u}) := \log \det\left( \bm{X}_s^\intercal \bm{U} \bm{X}_s \right) \tag{\textsc{D}$_s$} \label{Duals} \\
     \text{subject to} \quad & \sum_{i=1}^s u_i = 1, \quad \bm{u} \geq \bm{0}. \nonumber
\end{align}
\subsection{Computational complexity}
\label{sec:Computational_Complexity}
We compare the computational complexity of our algorithm to the current state of the art algorithm, the Wolfe-Atwood (WA) algorithm \cite{wolfe1970convergence,atwood1973sequences}, which computes a $\delta$-approximately optimal solution to \eqref{Dual}. It uses the Kumar-Yildirim (KY) initialisation \cite{kumar2005minimum}, which places equal weight on a small subset of points. These algorithms have time complexity $\mathcal{O} \left(n d^2 \left( \log\log d + \delta^{-1} \right) \right)$ \cite{todd2007khachiyan} and $\mathcal{O} \left( n d^2 \right)$ \cite{kumar2005minimum}, respectively. Together, this costs $\mathcal{O} \left(n d^2 \left( \log\log d + \delta^{-1} \right) \right)$.

The computational complexity of our approach is as follows. We calculate approximate leverage scores in $\mathcal{O} \left( nd \right)$ time, which, with high probability, satisfy $\hat{\ell}_i \left( \bm{X} \right) = \left( 1 \pm \mathcal{O} \left( \beta \right) \right) \ell_i \left( \bm{X} \right)$, for some $\beta \in (0,1)$ \cite{eshragh2023sequential}. Then Algorithm \ref{alg:DLSS_thresh} also runs in $\mathcal{O} \left( nd \right)$ time. We then use the WA algorithm with KY initialisation to find an optimal solution to \eqref{Duals}. Since $\bm{X}_s \in \text{\rm{I\!R}}^{s \times d}$, these algorithms have time complexity $\mathcal{O} \left(s d^2 \log \log d \right)$ and $\mathcal{O} \left( sd^2 \right)$, respectively. Thus, the total time complexity is $O\left( nd + sd^{2}\left( \log\log d + \delta^{-1} \right) \right)$. Further, if the leverage scores exhibit a power law decay, then $s = \text{poly}\left( d \left( 1 + \mathcal{O} \left( \beta \right) \right), \frac{1}{\left( 1 - \mathcal{O} \left( \beta \right) \right) \varepsilon} \right)$ (see Appendix \ref{app:dalss}). 
\subsection{Comparison with the work of Cohen et. al.}
Recently, Cohen et. al. \cite{cohen2019near} developed the Fixed Point algorithm, which computes a $\delta$-primal feasible solution to \eqref{Dual}. Woodruff and Yasuda \cite{woodruff2023new} extend this result; by sampling rows of $\bm{X}$ with probabilities proportionate to the weights calculated by the Fixed Point algorithm, they obtain a $\delta$-approximately optimal solution to \eqref{Dual}.

The computational complexity of the Fixed Point algorithm is as follows. We examine Algorithm 2 from \cite{cohen2019near}, which uses sketching techniques from randomised numerical linear algebra to speed up each iteration. Theorem C.7 from \cite{cohen2019near} states that Algorithm 2 takes at most $O\left( \delta^{-1} \log \left( \frac{n}{d} \right) \right)$ iterations to complete. In each iteration there are $O\left(\delta^{-1}\right)$ linear systems of the form $\bm{A}^\intercal \bm{W} \bm{A} \bm{x} = \bm{b}$ to be solved. Assuming $\bm{A}$ is dense, solving each of these linear systems costs $O(nd)$. This gives a total time complexity of $O\left( \delta^{-2} nd \log \left( \frac{n}{d} \right) \right)$ for dense matrices.

In big data regimes with $n\gg d$, it is reasonable to assume that $\delta < \frac{1}{d}$ is desirable. This assumption holds for many problem instances considered in the MVCE literature, since these problems typically have dimension $d \leq 200$, and are solved until tolerance $\delta = 10^{-6}$ (or stricter) is achieved (see, e.g. \cite{sun2004computation,damla2008linear,yu2011d,kallberg2019active,kudela2019minimum}). With this assumption, Algorithm 2 has total time complexity $O\left( nd^{3} \log (nd) \right)$ and our algorithm has total time complexity $O\left( nd + sd^{3} \right)$, where $s\leq n$. Thus, in the context of tall data matrices with dense input, our algorithm outperforms the Fixed Point algorithm theoretically. (This also holds numerically, see Appendix \ref{app:comparison}.)
\section{Initial optimality gap}
\label{sec:TheoreticalResults}
We would like to know how well MVCE$\left(\mathcal{X}_s \right)$ approximates MVCE$\left(\mathcal{X} \right)$. We will first provide an upper bound for the initial optimality gap, for a particular choice of initial $\bm{u}$. This $\bm{u}$ must be feasible for Problem \eqref{Duals}, that is, we must have $\sum_{i=1}^s u_i = 1$, and $\bm{u} \geq \bm{0}$. We choose
\begin{align*}
    \bm{u}_0 = \begin{bmatrix} \frac{1}{s}, \dots, \frac{1}{s} \end{bmatrix}^\intercal \in \text{\rm{I\!R}}^{s}.
\end{align*} 
Then the initial optimality gap is given by
\begin{align*}
    g^* - g_s( \bm{u}_0 ),
\end{align*}
where $g^*$ is the optimal objective value when the full set $\mathcal{X}$ is considered.

To derive our bound, we will compare our initial solution with an initialisation due to Khachiyan \cite{khachiyan1996rounding}. Khachiyan's initialisation $\bm{u}_K$ puts equal weight on all $n$ points of $\mathcal{X}$, and guarantees that
\begin{align*}
    g^* - g(\bm{u}_K) \leq d \log n,
\end{align*}
as shown by Khachiyan \cite{khachiyan1996rounding}. Therefore, our bound can be found by exploiting the fact that
\begin{align*}
    g^* - g_s(\bm{u}_0) &= (g^* - g(\bm{u}_K)) + (g(\bm{u}_K) - g_s(\bm{u}_0)).
\end{align*}
Upon simplification, 
\begin{align} \label{eq:g(uk)}
    g(\bm{u}_K) = \log \det \left( \bm{X}^\intercal \bm{X} \right) - d \log n,
\end{align}
and, similarly,
\begin{align} \label{eq:gs(u0)}
    g_s(\bm{u}_0) = \log\det\left( \bm{X}_s^\intercal \bm{X}_s \right) - d\log s.
\end{align}
Hence
\begin{align}\label{eq:initoptgap1}
    g^* - g_s(\bm{u}_0)
    &\leq \log \det \left( \bm{X}^\intercal \bm{X} \right) - \log\det\left( \bm{X}_s^\intercal \bm{X}_s \right) + d\log s.
\end{align}
This suggests the need to compare $\log\det\left( \bm{X}_s^\intercal \bm{X}_s \right)$ with $\log \det \left( \bm{X}^\intercal \bm{X} \right)$.

Suppose we construct $\bm{X}_s = \bm{R} \bm{X}$ as in Algorithm \ref{alg:DLSS_thresh}. Then Theorem \ref{thrm:thresh} guarantees that
\begin{align*}
    (1-\varepsilon) \bm{X}^\intercal \bm{X} \prec \bm{X}_s^\intercal \bm{X}_s.
\end{align*}
Now, for positive semidefinite matrices $\bm{A}$, $\bm{B}$, with $\bm{A} \preceq \bm{B}$, we have
\begin{align} \label{eq:det(A)<=det(B)}
    \det \bm{A} \leq \det \bm{B}.
\end{align}
Since the logarithm is monotonic, it then follows that $\log \det \bm{A} \leq \log \det \bm{B}$. Hence
\begin{align}\label{eq:initoptgap2}
      \log \det \left( \bm{X}^\intercal \bm{X} \right) 
      < \log \det \left( \bm{X}_s^\intercal \bm{X}_s \right) - d \log (1-\varepsilon).
\end{align}
Combining Equations \eqref{eq:initoptgap1} and \eqref{eq:initoptgap2}, we obtain the bound presented in Theorem \ref{thrm:initoptgap}.
\begin{theorem}\label{thrm:initoptgap}
    Select $\bm{X}_s = \bm{R} \bm{X}$ using Algorithm \ref{alg:DLSS_thresh}, and let our initial solution be given by $\bm{u}_0 = \frac{1}{s} \bm{e}$. Then
    \begin{align*}
        g^* - g_s(\bm{u}_0) < d\log\left(\frac{s}{1-\varepsilon} \right).
    \end{align*}
\end{theorem}
\begin{proof}
    See Appendix \ref{app:initoptgap}.
\end{proof}
\section{Final optimality gap} 
\label{sec:TheoreticalResults2}
We now provide an upper bound for the final optimality gap. The final optimality gap is given by
\begin{align*}
    g^* - g_s^*,
\end{align*}
where $g_s^*$ is the optimal objective value when only $\mathcal{X}_s$ is considered. Note that for any $\bm{u}$ feasible for \eqref{Duals}, we have
\begin{align*}
    g^* - g_s^* \leq g^* - g_s(\bm{u}),
\end{align*}
since $g_s$ is concave. 

Consider the feasible solution $\Tilde{\bm{u}}_s$ for \eqref{Duals}, given by
\begin{align*}
    \Tilde{\bm{u}}_s = \frac{1}{\bm{e}^\intercal \bm{u}_s} \bm{u}_s,
\end{align*}
where $\bm{u}_s$ contains the first $s$ entries of $\bm{u}^*$. We would like a bound similar to the one in Theorem \ref{thrm:thresh}, with $\bm{X}$ replaced with a rescaled version $\bm{Y}$. More precisely, let
\begin{align*}
    \bm{Y} := \sqrt{\bm{U}^*} \bm{X}.
\end{align*}
To apply Theorem \ref{thrm:thresh}, we require
\begin{align*}
    \sum_{i=s+1}^n \ell_i \left( \bm{Y} \right) < \varepsilon,
\end{align*}
for some $\varepsilon \in (0,1)$. The strategy is to show that
\begin{align} \label{eq:finaloptgap}
    \sum_{i=s+1}^n \ell_i \left( \bm{Y} \right) 
    \leq \sum_{i=s+1}^n \ell_i \left( \bm{X} \right),
\end{align}
which is less than $\varepsilon$ by construction of Algorithm \ref{alg:DLSS_thresh}. Therefore, we may apply Theorem \ref{thrm:thresh} with $\bm{Y}$ instead of $\bm{X}$, to obtain the bound
\begin{align*}
    (1-\varepsilon) \bm{X}^\intercal \bm{U}^* \bm{X} \prec \bm{X}_s^\intercal \bm{U}_s \bm{X}_s.
\end{align*}
That is, the feasible solution $\Tilde{\bm{u}}_s$ satisfies
\begin{align*}
    (1-\varepsilon) \bm{X}^\intercal \bm{U}^* \bm{X} 
    \preceq \frac{1-\varepsilon}{\bm{e}^\intercal \bm{u}_s} \bm{X}^\intercal \bm{U}^* \bm{X} 
    \prec \frac{1}{\bm{e}^\intercal \bm{u}_s} \bm{X}_s^\intercal \bm{U}_s \bm{X}_s 
    = \bm{X}_s^\intercal \Tilde{\bm{U}}_s \bm{X}_s.
\end{align*}
Hence
\begin{align*}
    g^* - g_s^* 
    &\leq g^* - g_s(\Tilde{\bm{u}}_s) \\
    &= \log \det \left( \bm{X}^\intercal \bm{U}^* \bm{X} \right) - \log \det \left( \bm{X}_s^\intercal \Tilde{\bm{U}}_s \bm{X}_s \right) \\
    &< \log \det \left( \bm{X}^\intercal \bm{U}^* \bm{X} \right) - \left( d \log (1-\varepsilon) + \log \det \left( \bm{X}^\intercal \bm{U}^* \bm{X} \right) \right) \\
    &= d \log \left( \frac{1}{1-\varepsilon} \right).
\end{align*}

\begin{theorem} \label{thrm:finaloptgap}
    Use Algorithm \ref{alg:DLSS_thresh} to construct $\bm{X}_s = \bm{R} \bm{X}$, with $\varepsilon \in (0,1)$. If there exists an optimal solution $\bm{u}^*$ of \eqref{Dual} satisfying $u_i^* > 0$ for all $i = 1,\dots,s$, then
    \begin{align*}
        g^* - g_s^*
        < d \log \left( \frac{1}{1 - \varepsilon} \right).
    \end{align*}
    Otherwise,
    \begin{align*}
        g^* - g_s^*
        < d \log \left( \frac{1 + \delta}{1 - \varepsilon} \right),
    \end{align*}
    where the parameter $\delta > 0$ can be chosen to be arbitrarily small.
\end{theorem}
\begin{proof}
    See Appendix \ref{app:finaloptgap}.
\end{proof}
\section{Numerical results} 
\label{sec:empirical_results}
We generate three large datasets of size $n = 10^7$, $d~=~100$. Without loss of generality, we assume the leverage scores are sorted in descending order, that is, $\ell_1 \left( \bm{X} \right) \geq \dots \geq \ell_n \left( \bm{X} \right)$. The first dataset is Rotated Cauchy \cite{todd2016book}. The points are generated so that they have rotational symmetry, and the distances of the points from the origin are Cauchy. The leverage scores of these points quickly decay. The second dataset is Lognormal, which has a shallower leverage score decay. The third dataset is Gaussian, which has leverage scores that are close to uniform. We provide additional numerical results in Appendix \ref{app:numresults}.

All computations are performed on a personal laptop with a 64 bit MacOS 13 operating system, and a 2.4 GHz Quad-Core Intel Core i5 processor with 8 GB of RAM. The algorithms are run using MATLAB (R2021a).

We use the WA algorithm \cite{wolfe1970convergence,atwood1973sequences} to calculate the MVCEs of the two datasets, with $\delta = 10^{-9}$. We initialise using the KY algorithm \cite{kumar2005minimum}. We use Todd's Matlab implementation of these algorithms \cite{todd2016minvol}. Then, for $s$ varying from $0.1$ to 10$\%$ of $n$, we sample from each dataset in three ways: deterministic leverage score sampling, uniform sampling, and randomised leverage score sampling (sampling with probabilities proportionate to the leverage scores \cite{drineas2008relative}). In these results, we use the exact leverage scores.

Let $\Tilde{g}^*$ and $\Tilde{g}_s^*$ be the optimal values obtained when using the WA algorithm on the full and sampled datasets respectively. In Figure \ref{fig:rotcauchy_optgap}, we summarise the calculated optimality gaps $\Tilde{g}^* - \Tilde{g}_s^*$ for the Rotated Cauchy dataset. The deterministic and randomised leverage score sampling performed similarly, with near zero optimality gap for all values of $s$. Uniform sampling performed poorly, its optimality gap decreasing with increasing $s$. In Figure \ref{fig:rotcauchy_time} we summarise the total computation time for calculating the MVCEs. For the sampled datasets, this also includes the computation time for calculating leverage scores (if applicable), and sampling from the dataset. The total computation time for the full dataset was $2450$ seconds, that is, just over $40$ minutes. Uniform sampling was faster than the leverage score sampling methods (due to the calculation of the leverage scores), but had very large optimality gaps (see Figure \ref{fig:rotcauchy_optgap}).

In Figure \ref{fig:lognormal_optgap}, we summarise the calculated optimality gaps $\Tilde{g}^* - \Tilde{g}_s^*$ for the Lognormal dataset. The deterministic and randomised leverage score sampling performed similarly, with zero optimality gap for all values of $s$. Uniform sampling performed poorly, its optimality gap decreasing with increasing $s$. In Figure \ref{fig:lognormal_time} we summarise the total computation time for calculating the MVCEs. The total computation time for the full dataset was $2564$ seconds, that is, just over $42$ minutes. Uniform sampling was generally faster than the leverage score sampling methods (due to the calculation of the leverage scores), but had very large optimality gaps (see Figure \ref{fig:lognormal_optgap}).

In Figure \ref{fig:gaussian_optgap}, we summarise the calculated optimality gaps $\Tilde{g}^* - \Tilde{g}_s^*$ for the Gaussian dataset. Only the deterministic leverage score sampling performed well, with near zero optimality gap for all $s$. The uniform and randomised leverage score sampling performed similarly, with optimality gaps decreasing with increasing $s$. This is unsurprising, since the leverage scores for a Gaussian dataset are close to uniform. In Figure \ref{fig:gaussian_time} we summarise the total computation time for calculating the MVCEs. The total computation time for the full dataset was very large, at $141840$ seconds, that is, over $39$ hours. As with the other dataset, uniform sampling was the fastest, but at the cost of a larger optimality gap (see Figure \ref{fig:gaussian_optgap}). Additionally, both leverage score sampling methods had similar runtimes, but only the deterministic sampling had near zero optimality gap (see Figure \ref{fig:gaussian_optgap}).

Overall, the deterministic leverage score sampling performs the best, achieving both a small optimality gap and greatly decreasing computation time on all three datasets.

\begin{figure}[p]
\vskip 0.1in
\begin{center}
\begin{subfigure}{0.49\textwidth}
\includegraphics[width=\linewidth,trim=0.1cm 0.1cm 1cm 0.2cm,clip]{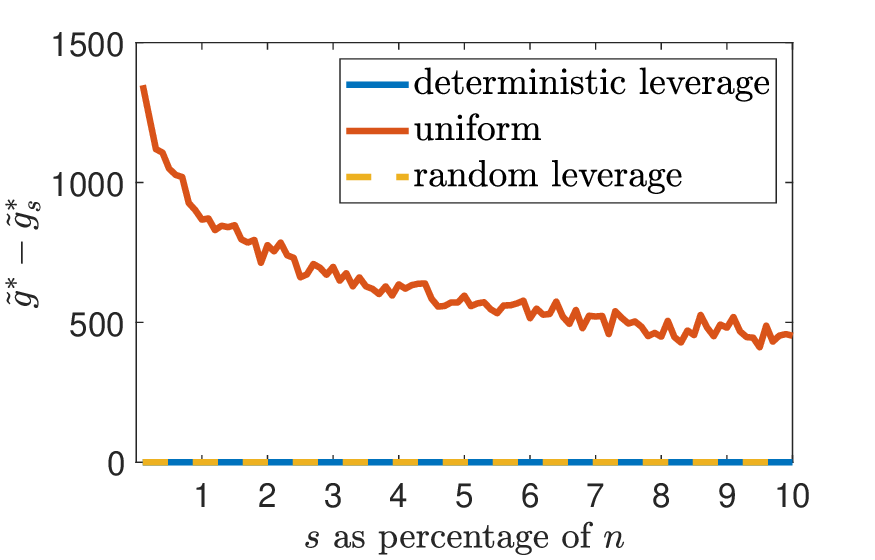}
    \caption{}
    \label{fig:rotcauchy_optgap}
\end{subfigure}
\begin{subfigure}{0.49\textwidth}
\includegraphics[width=\linewidth,trim=0.1cm 0.1cm 1cm 0.2cm,clip]{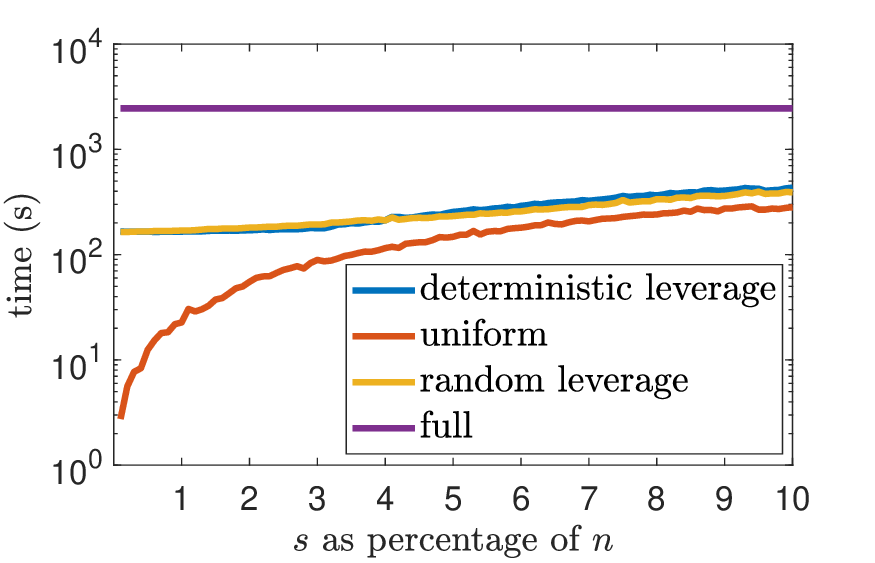}
    \caption{}
    \label{fig:rotcauchy_time}
\end{subfigure}
\caption{Rotated Cauchy: Calculated optimality gap summary (a) and time summary (b).}
\label{fig:rotcauchy}
\end{center}
\vskip -0.1in
\end{figure}

\begin{figure}[p]
\vskip 0.1in
\begin{center}
\begin{subfigure}{0.49\textwidth}
\includegraphics[width=\linewidth,trim=0.1cm 0.1cm 1cm 0.2cm,clip]{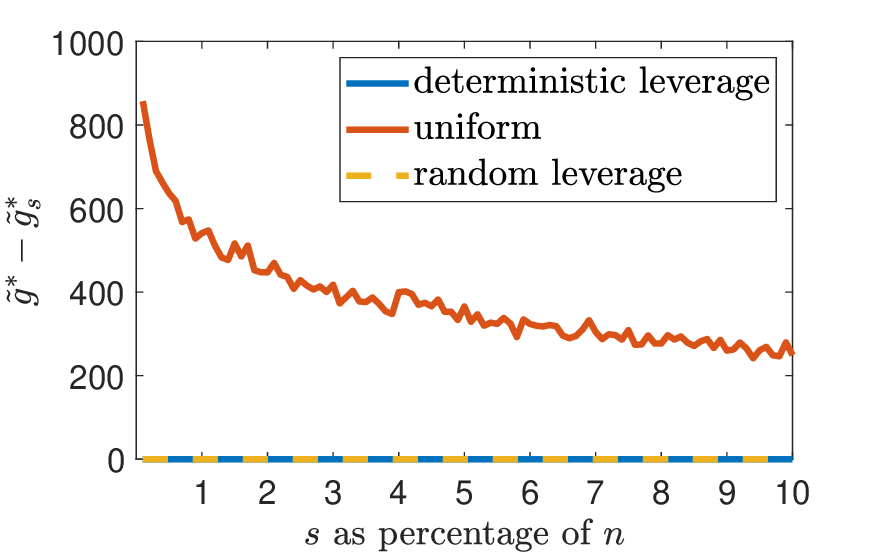}
    \caption{}
    \label{fig:lognormal_optgap}
\end{subfigure}
\begin{subfigure}{0.49\textwidth}
\includegraphics[width=\linewidth,trim=0.1cm 0.1cm 1cm 0.2cm,clip]{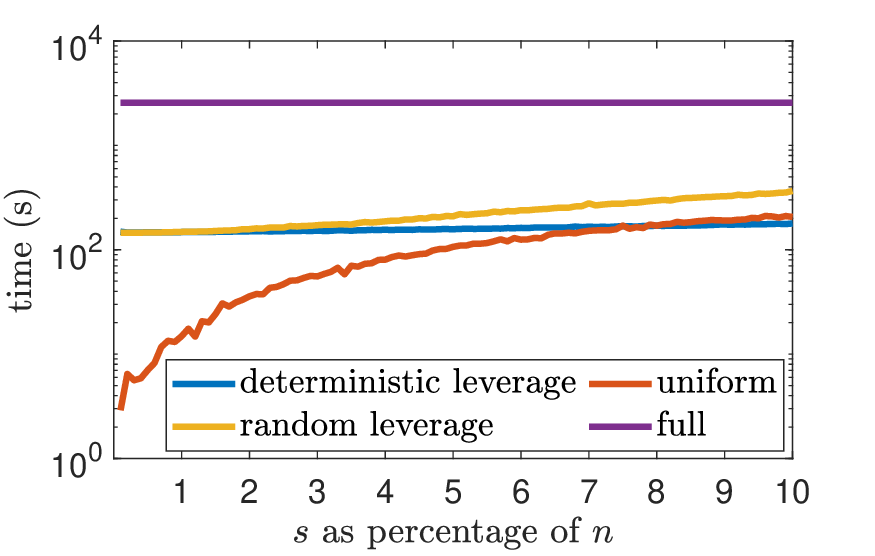}
    \caption{}
    \label{fig:lognormal_time}
\end{subfigure}
\caption{Lognormal: Calculated optimality gap summary (a) and time summary (b).}
\label{fig:lognormal}
\end{center}
\vskip -0.1in
\end{figure}

\begin{figure}[p]
\vskip 0.1in
\begin{center}
\begin{subfigure}{0.49\textwidth}
\includegraphics[width=\linewidth,trim=0.1cm 0.1cm 1cm 0.2cm,clip]{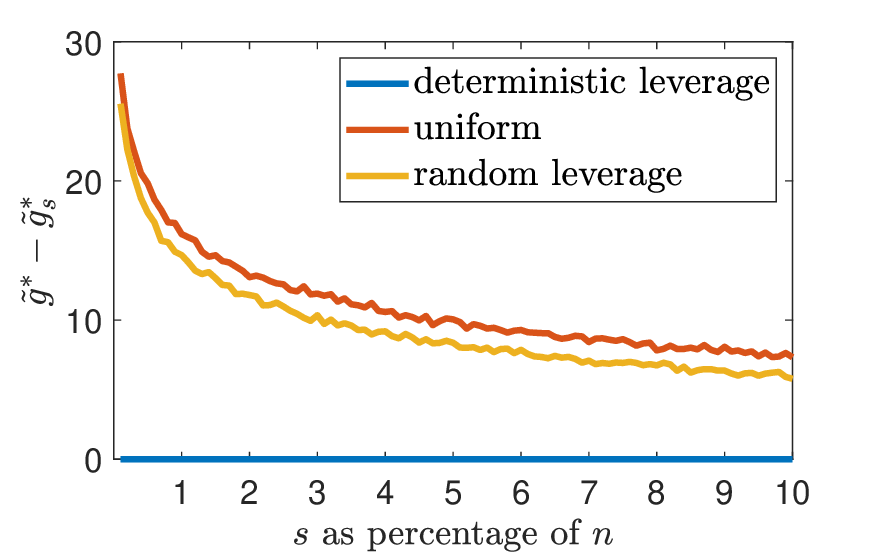}
    \caption{}
    \label{fig:gaussian_optgap}
\end{subfigure}
\begin{subfigure}{0.49\textwidth}
\includegraphics[width=\linewidth,trim=0.1cm 0.1cm 1cm 0.2cm,clip]{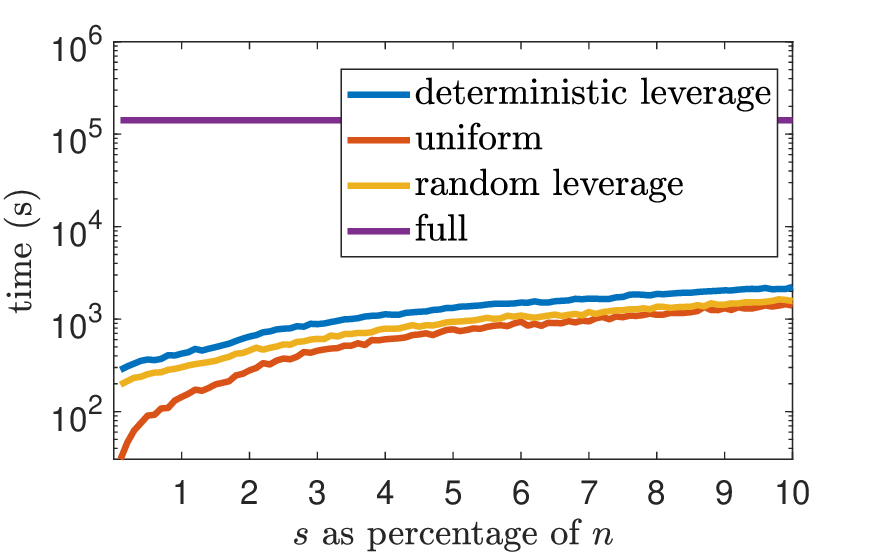}
    \caption{}
    \label{fig:gaussian_time}
\end{subfigure}
\caption{Gaussian: Calculated optimality gap summary (a) and time summary (b).}
\label{fig:gaussian}
\end{center}
\vskip -0.1in
\end{figure}
\section{Conclusion} \label{sec:Conclusion}
In this paper, we have provided the first theoretical guarantees on the quality of initial and final solutions for the MVCE problem, when sampling points deterministically according to their statistical leverage scores. We proved this approach is efficient, assuming the leverage scores exhibit a power law decay. Numerical results show that our data-driven leverage score sampling algorithm performs even better than the established theoretical error bounds, even in cases where the leverage scores are close to uniform distribution, which could be a by-product of our analysis. Future work could include extending these results to other data-driven sampling methods, including randomised leverage score sampling, and also implementing this algorithm extensively on real-world large-scale datasets.
\bibliographystyle{plain}

\newpage
\appendix
\section{Alternative proof for Theorem \ref{thrm:thresh}} \label{app:altproof}
We provide a proof of Theorem \ref{thrm:thresh}, including a new simplified proof of the lower bound.

\begin{theorem} \label{thrm:thresh2}
    Let $\varepsilon \in (0,1).$ Suppose the rows of $\bm{X} \in \text{\rm{I\!R}}^{n \times d}$ are ordered such that $\sum_{i=s+1}^n \ell_i \left( \bm{X} \right) < \varepsilon$, for some $s < n$. Let $\bm{X}_s = \bm{R}\bm{X}$ contain the first $s$ rows of $\bm{X}$. Then
    \begin{align*}
        (1-\varepsilon) \bm{X}^\intercal \bm{X} \prec \bm{X}_s^\intercal \bm{X}_s \preceq \bm{X}^\intercal \bm{X}.
    \end{align*}
\end{theorem}
\begin{proof}
    We consider the upper bound $\left( \bm{R}\bm{X} \right)^\intercal \bm{R}\bm{X} \preceq \bm{X}^\intercal \bm{X}$ first. For any sampling matrix $\bm{R}$, we have that
    \begin{align*}
        \bm{I} - \bm{R}^\intercal \bm{R} \succeq \bm{0}.
    \end{align*}
    Now, the conjugation rule states that if $\bm{A}$ and $\bm{B}$ are symmetric positive semidefinite matrices with $\bm{A} \preceq \bm{B}$, and $\bm{C}$ is any real matrix of compatible dimension, then $\bm{C} \bm{A} \bm{C}^\intercal \preceq \bm{C} \bm{B} \bm{C}^\intercal$. Using this, we obtain
    \begin{align*}
        \bm{0} 
        \preceq \bm{X}^\intercal \left( \bm{I} - \bm{R}^\intercal \bm{R} \right) \bm{X}
        = \bm{X}^\intercal \bm{X} - \left(\bm{R} \bm{X}\right)^\intercal \bm{R} \bm{X},
    \end{align*}
    which rearranges to give our upper bound.
    
    We now prove the lower bound. We will use the fact that for any $i$,
    \begin{align*}
        \bm{x}_i \bm{x}_i^\intercal \preceq \ell_i \left( \bm{X} \right) \bm{X}^\intercal \bm{X},
    \end{align*}
    see, for example, the proof of Lemma 4 in \cite{cohen2015uniform}. Hence
    \begin{align*}
        \bm{X}^\intercal \bm{X} - \left(\bm{R} \bm{X}\right)^\intercal \bm{R} \bm{X}
        = \sum_{i=s+1}^n \bm{x}_i \bm{x}_i^\intercal 
        \preceq \sum_{i=s+1}^n \ell_i \left( \bm{X} \right) \bm{X}^\intercal \bm{X} 
        = \left(\sum_{i=s+1}^n \ell_i \left( \bm{X} \right) \right) \bm{X}^\intercal \bm{X} 
        \prec \varepsilon \bm{X}^\intercal \bm{X},
    \end{align*}
    since $\sum_{i=s+1}^n \ell_i \left( \bm{X} \right) < \varepsilon$. Rearranging, we obtain our lower bound. 
\end{proof}
\section{Deterministic approximate leverage score sampling} \label{app:dalss}
The computation time in Algorithm \ref{alg:DLSS_thresh} is dominated by the cost of calculating the leverage scores exactly, which is $\mathcal{O} \left( nd^2 \right)$. We can instead use approximate leverage scores, which are more computationally efficient to calculate. We demonstrate the approach using the algorithm of Eshragh et al. \cite{eshragh2023sequential}.

\begin{theorem}[\cite{eshragh2023sequential}, Theorem 4] \label{thrm:approxlev}
    Fix a constant $\beta \in (0,1)$. Let the leverage scores of the rows of a matrix $\bm{X} \in \text{\rm{I\!R}}^{n \times d}$ be given by
    $\ell_1 \left( \bm{X} \right), \dots, \ell_n \left( \bm{X} \right)$. Then there exists a randomised algorithm that calculates approximate leverage scores $\hat{\ell}_1 \left( \bm{X} \right), \dots, \hat{\ell}_n \left( \bm{X} \right)$ such that with high probability, simultaneously for all $i = 1,\dots,n$, $\hat{\ell}_i \left( \bm{X} \right) = \left( 1 \pm \mathcal{O} \left( \beta \right)\right) \ell_i \left( \bm{X} \right)$. This algorithm has time complexity $\mathcal{O} \left( nd \right)$.
\end{theorem}

We then sample deterministically according to these approximate leverage scores. We summarise the modified sampling procedure in Algorithm \ref{alg:DALSS_thresh}. For simplicity, assume we can write $\hat{\ell}_i \left( \bm{X} \right) = \left( 1 \pm \alpha \right) \ell_i \left( \bm{X} \right)$, for some $\alpha \in (0,1)$.

\begin{algorithm}
    \begin{algorithmic}[1]
        \REQUIRE $\bm{X} = \begin{bmatrix} \bm{x}_1, \dots, \bm{x}_n \end{bmatrix}^\intercal \in \text{\rm{I\!R}}^{n \times d}, \alpha \in (0,1), \varepsilon \in (0,1)$
        \STATE Calculate approximate leverage scores for each row in $\bm{X}$, using Theorem \ref{thrm:approxlev}. Without loss of generality, let $\hat{\ell}_1 \left( \bm{X} \right) \geq \dots \geq \hat{\ell}_n \left( \bm{X} \right)$.
        \STATE Let $t = \sum_{i=1}^n \hat{\ell}_i \left( \bm{X} \right)$.
        \STATE Let $s = \arg \min_j \left( \sum_{i=1}^j \hat{\ell}_i \left( \bm{X} \right) \geq t - \left( 1 - \alpha \right) \varepsilon \right)$.
        \STATE Let $\bm{R} = \bm{0} \in \text{\rm{I\!R}}^{s \times n}$.
        \FOR{$i = 1:s$}
            \STATE Set row $i$ of $\bm{R}$ equal to $\bm{e}_i$.
        \ENDFOR
        \ENSURE $\bm{R}, s$
        \caption{Deterministic Approximate Leverage Score Sampling with Threshold}
        \label{alg:DALSS_thresh}
    \end{algorithmic}
\end{algorithm}

We have chosen $s$ carefully, to ensure the following subspace embedding result.

\begin{corollary}\label{corr:analogous}
    Let $\alpha \in (0,1)$, and $\varepsilon \in (0,1)$. Use Algorithm \ref{alg:DALSS_thresh} to construct $\bm{X}_s = \bm{R} \bm{X}$. Then, with high probability, we have
    \begin{align*}
        (1-\varepsilon) \bm{X}^\intercal \bm{X} \prec \bm{X}_s^\intercal \bm{X}_s \preceq \bm{X}^\intercal \bm{X}.
    \end{align*}
\end{corollary}
\begin{proof}
    For this proof, we only need to show that $\sum_{i=s+1}^n \ell_i \left( \bm{X} \right) < \varepsilon$ holds with high probability. We can then apply Theorem \ref{thrm:thresh2} to obtain our bound.

    Assume that $\hat{\ell}_i \left( \bm{X} \right) \geq \left( 1 - \alpha \right) \ell_i \left( \bm{X} \right)$ holds simultaneously for all $i = 1,\dots,n$. By construction of Algorithm \ref{alg:DALSS_thresh}, we have
    \begin{align*}
        \sum_{i=s+1}^n \hat{\ell}_i \left( \bm{X} \right)
        < \left( 1 - \alpha \right) \varepsilon.
    \end{align*}
    Then since $\hat{\ell}_i \left( \bm{X} \right) \geq \left( 1 - \alpha \right) \ell_i \left( \bm{X} \right)$, for all $i = 1,\dots,n$, we have
    \begin{align*}
        \left( 1 - \alpha \right) \varepsilon
        > \sum_{i=s+1}^n \hat{\ell}_i \left( \bm{X} \right)
        \geq \sum_{i=s+1}^n \left( 1 - \alpha \right) \ell_i \left( \bm{X} \right)
        = \left( 1 - \alpha \right) \sum_{i=s+1}^n \ell_i \left( \bm{X} \right),
    \end{align*}
    that is,
    \begin{align*}
        \sum_{i=s+1}^n \ell_i \left( \bm{X} \right) < \varepsilon.
    \end{align*}
\end{proof}

\begin{corollary}\label{cor:sbound}
    If the leverage scores exhibit a power law decay, then $s = \text{poly}\left( d \left( 1 + \alpha \right), \frac{1}{\left( 1 - \alpha \right) \varepsilon} \right)$.
\end{corollary}
\begin{proof}
    Papailiopoulos et al. \cite{papailiopoulos2014provable} show that if the leverage scores exhibit a power law decay, then Algorithm \ref{alg:DLSS_thresh} outputs
    \begin{align*}
        s 
        = \arg \min_j \left( \sum_{i=1}^j \ell_i \left( \bm{X} \right) 
        > d - \varepsilon \right) 
        = \max \left\lbrace \left( \frac{2d}{\varepsilon} \right)^{\frac{1}{1+\eta}} - 1, \left( \frac{2d}{\eta \varepsilon} \right)^{\frac{1}{\eta}} - 1, d \right\rbrace
        = \text{poly}\left( d, \frac{1}{ \varepsilon} \right).
    \end{align*}
    Then, comparing Algorithm \ref{alg:DLSS_thresh} with Algorithm \ref{alg:DALSS_thresh}, we need only replace $d$ with $t$, and $\varepsilon$ with $\left( 1 - \alpha \right) \varepsilon$. That is, Algorithm \ref{alg:DALSS_thresh} outputs
    \begin{align*}
        s 
        = \arg \min_j \left( \sum_{i=1}^j \hat{\ell}_i \left( \bm{X} \right) > t - \left( 1 - \alpha \right) \varepsilon \right) 
        = \text{poly}\left( t, \frac{1}{ \left( 1 - \alpha \right) \varepsilon} \right).
    \end{align*}
    But
    \begin{align*}
        t
        = \sum_{i=1}^n \hat{\ell}_i \left( \bm{X} \right) 
        \leq \sum_{i=1}^n \left( 1 + \alpha \right) \ell_i \left( \bm{X} \right) 
        = \left( 1 + \alpha \right) \sum_{i=1}^n \ell_i \left( \bm{X} \right) 
        = \left( 1 + \alpha \right) d.
    \end{align*}
    Hence $s = \text{poly}\left( d \left( 1 + \alpha \right), \frac{1}{\left( 1 - \alpha \right) \varepsilon} \right)$.
\end{proof}
\section{Missing details in proof of Theorem \ref{thrm:initoptgap}} \label{app:initoptgap}
\subsection{Proof of Equations \eqref{eq:g(uk)} and \eqref{eq:gs(u0)}}
We have that
\begin{align*}
    g(\bm{u}_K)
    &= \log \det \left( \bm{X}^\intercal \bm{U}_K \bm{X} \right) \\
    &= \log \det \left( \frac{1}{n} \bm{X}^\intercal \bm{X} \right) \\
    &= \log \left( \frac{1}{n^d} \det \left( \bm{X}^\intercal \bm{X} \right) \right) \\
    &= \log \det \left( \bm{X}^\intercal \bm{X} \right) - d \log n,
\end{align*}
and, similarly,
\begin{align*}
    g_s(\bm{u}_0)
    &= \log \det \left( \bm{X}_s^\intercal \bm{U}_0 \bm{X}_s \right) \\
    &= \log \det \left( \frac{1}{s} \bm{X}_s^\intercal \bm{X}_s \right) \\
    &= \log \det \left( \bm{X}_s^\intercal \bm{X}_s \right) - d \log s.
\end{align*}
\subsection{Proof of Inequality \eqref{eq:det(A)<=det(B)}}
We first need to prove the following lemma.
\begin{lemma}\label{lem:eigenvalues}
    Suppose we have two $d \times d$ symmetric positive semidefinite matrices $\bm{A}$ and $\bm{B}$, with $\bm{A} \preceq \bm{B}$. Let the eigenvalues of $\bm{A}$ be given by $\lambda_1\left( \bm{A} \right) \leq \dots \leq \lambda_d\left( \bm{A} \right)$, and the eigenvalues of $\bm{B}$ be given by $\lambda_1\left( \bm{B} \right) \leq \dots \leq \lambda_d\left( \bm{B} \right)$. Then for all $i = 1,\dots,d$, we have
    \begin{align*}
        \lambda_i\left(\bm{A}\right) \leq \lambda_i\left(\bm{B}\right).
    \end{align*}
\end{lemma}
\begin{proof}
    Use Min-max Theorem, along with the fact that $\bm{x}^\intercal \bm{A} \bm{x} \leq \bm{x}^\intercal \bm{B} \bm{x}$ for all $\bm{x}$.
\end{proof}

We are now ready to prove the inequality.
\begin{lemma}\label{lem:determinants}
    Suppose we have two $d \times d$ symmetric positive semidefinite matrices $\bm{A}$ and $\bm{B}$, with $\bm{A} \preceq \bm{B}$. Then
    \begin{align*}
        \det\left( \bm{A} \right) \leq \det\left( \bm{B} \right).
    \end{align*}
\end{lemma}
\begin{proof}
    We have that
    \begin{align*}
        \log\det\left( \bm{A} \right) 
        &= \log \left( \prod_{i=1}^d \lambda_i\left( \bm{A} \right) \right) \\ 
        &= \sum_{i=1}^d \log \lambda_i\left( \bm{A} \right) \\
        &\leq \sum_{i=1}^d \log \lambda_i\left( \bm{B} \right) \\
        &= \log \left( \prod_{i=1}^d \lambda_i\left( \bm{B} \right) \right) \\
        &= \log\det\left( \bm{B} \right),
    \end{align*}
    where the inequality uses Lemma \ref{lem:eigenvalues} and the fact that $\log$ is a monotonically increasing function. Then
    \begin{align*}
        \det(\bm{A}) 
        = e^{\log \det\left( \bm{A} \right)}
        \leq e^{\log \det\left( \bm{B} \right)}
        = \det(\bm{B}),
    \end{align*}
    where the inequality is because the exponential function is monotonically increasing.
\end{proof}
\section{Final optimality gap} \label{app:finaloptgap}
We break this proof into two cases:
\begin{enumerate}
    \item $u_i^* > 0$, for all $i = 1,\dots,s$, and
    \item $u_i^* \geq 0$, for all $i = 1,\dots,s$.
\end{enumerate}
We do this because the proof of Inequality \eqref{eq:finaloptgap} relies on scaling $\sqrt{\bm{U}^*}$ so that the first $s$ diagonal terms are $\geq 1$, and the last $n-s$ diagonal terms are $\leq 1$. Having any zeros in the first $s$ terms prevents such a scaling.

We consider Case 1 first. First, follow the proof outline from Section \ref{sec:TheoreticalResults2}. Then, what remains to be shown is that Inequality \eqref{eq:finaloptgap} holds.
\subsection{Case 1: $u_i^* > 0,$ for all $i = 1,\dots,s$}
The following lemma will be useful for proving Inequality \eqref{eq:finaloptgap}. It shows the change in leverage scores when one row of $\bm{X}$ is scaled.

\begin{lemma}[\cite{ordozgoiti2022generalized}, Lemma 3.2] \label{lem:scaledrowlev}
    Let $\bm{X} \in \text{\rm{I\!R}}^{n \times d}$ have rank $d$. Let $\alpha \geq 0$. Define $\bm{\Sigma}^{(i)}$ as the diagonal matrix satisfying $\bm{\Sigma}^{(i)}_{i\,i} = \alpha$, $\bm{\Sigma}^{(i)}_{j\,j} = 1$, if $j \neq i$. Then 
    \begin{align*}
        \ell_j\left( \bm{\Sigma}^{(i)} \bm{X} \right)
        = \ell_j\left( \bm{X} \right) - \frac{\left( \alpha^2 - 1 \right) \left( \bm{x}_j^\intercal \left( \bm{X}^\intercal \bm{X} \right)^\intercal \bm{x}_i \right)^2}{1 + \left(\alpha^2 - 1 \right) \ell_i \left( \bm{X} \right)},
    \end{align*}
    and, in particular,
    \begin{align*}
        \ell_i\left( \bm{\Sigma}^{(i)} \bm{X} \right) 
        = \frac{\alpha^2 \ell_i \left( \bm{X} \right)}{1 + \left(\alpha^2 - 1 \right) \ell_i \left( \bm{X} \right)}.
    \end{align*}
\end{lemma}

We now ready to prove the inequality.
\begin{proposition} \label{prop:levY1}
    Use Algorithm \ref{alg:DLSS_thresh} to construct $\bm{X}_s = \bm{R} \bm{X}$, with $\varepsilon \in (0,1)$. Suppose that an optimal solution $\bm{u}^*$ of \eqref{Dual} satisfies $u_i^* > 0$ for all $i = 1,\dots,s$. Define $\bm{Y} := \sqrt{\bm{U}^*} \bm{X}$. Then 
    \begin{align*}
        \sum_{i=s+1}^n \ell_i \left( \bm{Y} \right) 
        \leq \sum_{i=s+1}^n \ell_i \left( \bm{X} \right).
    \end{align*}
\end{proposition}
\begin{proof}
    The structure of this proof is based on the proof of Theorem 3.2 in \cite{ordozgoiti2022generalized}. We will show that the sum of the last $n-s$ leverage scores of $\bm{X}$ do not increase when we premultiply it by $\sqrt{\bm{U}^*}$. Without loss of generality, we assume that $\min_{i=1,...,s} \sqrt{u_i^*} \geq \max_{i=s+1,...,n} \sqrt{u_i^*}$. This is because, for any $k$, we can write
    \begin{align*}
        u_k^* \bm{x}_k \bm{x}_k^\intercal 
        = \sum_{i=1}^\beta \frac{u_k^*}{\beta} \bm{x}_k \bm{x}_k^\intercal,
    \end{align*}
    where $\beta \in \text{\rm{I\!N}}$, such that $\frac{u_k^*}{\beta}$ is sufficiently small.

    First, we scale $\sqrt{\bm{U}^*}$ so that the first $s$ diagonal terms are $\geq 1$, and the last $n-s$ diagonal terms are $\leq 1$. Let $\alpha = 1/\min_{i=1,...,s} \sqrt{u_i^*}$. We define
    \begin{align*}
        \bm{V} := \alpha \sqrt{\bm{U}^*}.
    \end{align*}
    We exploit the fact that
    \begin{align*}
        \ell_i \left( \bm{V} \bm{X} \right)
        &= \ell_i \left( \alpha \sqrt{\bm{U}^*} \bm{X} \right) \\
        &= \ell_i \left( \alpha \bm{Y} \right) \\
        &= \left( \left( \alpha \bm{Y} \right) \left( \left( \alpha \bm{Y} \right)^\intercal \left( \alpha \bm{Y} \right) \right)^{-1} \left( \alpha \bm{Y} \right)^\intercal \right)_{i\,i} \\
        &= \left( \alpha^2 \frac{1}{\alpha^2} \bm{Y} \left( \bm{Y}^\intercal \bm{Y} \right)^{-1} \bm{Y}^\intercal \right)_{i\,i} \\
        &= \ell_i \left( \bm{Y} \right).
    \end{align*}
    To see how $\bm{V}$ affects the leverage scores of $\bm{X}$, we will consider scaling each row separately. We define $\bm{V}^{(i)}$ as the diagonal matrix satisfying $\bm{V}_{i\,i}^{(i)} = \bm{V}_{i\,i} = \alpha \sqrt{u_i^*}$ and $\bm{V}_{j\,j}^{(i)} = 1$, if $j \neq i$. Now, we consider the leverage scores of $\bm{V}^{(i)} \bm{X}$. From Lemma \ref{lem:scaledrowlev}, we have
    \begin{enumerate}
        \item If $i \leq s$, then $\alpha \sqrt{u_i^*} \geq 1$. Hence $\ell_i \left( \bm{V}^{(i)} \bm{X} \right) \geq \ell_i \left( \bm{X} \right)$, and $\ell_j \left( \bm{V}^{(i)} \bm{X} \right) \leq \ell_j \left( \bm{X} \right),$ if $j \neq i$.
        \item If $i \geq s$, then $\alpha \sqrt{u_i^*} \leq 1$. Hence $\ell_i \left( \bm{V}^{(i)} \bm{X} \right) \leq \ell_i \left( \bm{X} \right)$, and $\ell_j \left( \bm{V}^{(i)} \bm{X} \right) \geq \ell_j \left( \bm{X} \right),$ if $j \neq i$.
    \end{enumerate}
    Now, we consider scaling the first $s$ rows. From the previous discussion, we can conclude that the leverage scores of the last $n-s$ rows do not increase. Therefore, the sum of the first $s$ leverage scores does not decrease. (This is because the sum of the leverage scores remains constant, since scaling rows by non-zero constants does not affect the rank.) Next, consider scaling the last $n-s$ rows. From the discussion above, all the leverage scores of the first $s$ rows cannot decrease. That is, the sum of the first $s$ leverage scores again does not decrease. We conclude that the sum of the last $n-s$ leverage scores cannot increase. (This is because the sum of the leverage scores remains constant. Suppose $\text{rank}\left( \bm{V}\bm{X} \right) < d$. Then $\text{rank}\left( \bm{X}^\intercal \bm{U}^* \bm{X} \right) = \text{rank}\left( \left( \bm{V}\bm{X} \right)^\intercal \left( \bm{V}\bm{X} \right) \right) = \text{rank}\left( \bm{V}\bm{X} \right) < d$, which is a contradiction, since $\bm{X}^\intercal \bm{U}^* \bm{X}$ is invertible.) Hence
    \begin{align*}
        \sum_{i=1}^n \ell_i \left( \bm{Y} \right) 
        = \sum_{i=1}^n \ell_i \left( \bm{V} \bm{X} \right)
        \leq \sum_{i=1}^n \ell_i \left( \bm{X} \right).
    \end{align*}
\end{proof}
\subsection{Case 2: $u_i^* \geq 0,$ for all $i = 1,\dots,s$}
This time, we will examine a $\delta$-feasible solution $\bm{u}$ for \eqref{Dual}, which has $u_i > 0$ for all $i=1,\dots,n$. Such a solution can be achieved by using the Frank-Wolfe algorithm \cite{frank1956algorithm,wolfe1970convergence} with Khachiyan's initialisation \cite{khachiyan1996rounding}. 

Consider the feasible solution $\Tilde{\bm{u}}_s$ for \eqref{Duals}, given by
\begin{align*}
    \Tilde{\bm{u}}_s = \frac{1}{\bm{e}^\intercal \bm{u}_s} \bm{u}_s,
\end{align*}
where $\bm{u}_s$ contains the first $s$ entries of $\bm{u}$. We would like a bound similar to the one in Theorem \ref{thrm:thresh2}, with $\bm{X}$ replaced with
\begin{align*}
    \bm{Y} := \sqrt{\bm{U}} \bm{X}.
\end{align*}
We require the following result.

\begin{corollary}
    Use Algorithm \ref{alg:DLSS_thresh} to construct $\bm{X}_s = \bm{R} \bm{X}$, with $\varepsilon \in (0,1)$. Let $\bm{u}$ be a $\delta$-feasible solution to \eqref{Dual}. Then 
    \begin{align*}
        \sum_{i=s+1}^n \ell_i \left( \bm{Y} \right) 
        \leq \sum_{i=s+1}^n \ell_i \left( \bm{X} \right).
    \end{align*}
\end{corollary}
\begin{proof}
    Follow the proof of Proposition \ref{prop:levY1}, but replace $\bm{u}^*$ with $\bm{u}$, and redefine $\bm{Y} := \sqrt{\bm{U}} \bm{X}$.
\end{proof}

Then, by the construction of Algorithm \ref{alg:DLSS_thresh}, we have
\begin{align*}
    \sum_{i=s+1}^n \ell_i \left( \bm{Y} \right) 
    \leq \sum_{i=s+1}^n \ell_i \left( \bm{X} \right)
    < \varepsilon,
\end{align*}
for some $\varepsilon \in (0,1)$. Therefore, we may apply Theorem \ref{thrm:thresh2} with $\bm{Y}$ instead of $\bm{X}$, to obtain the bound
\begin{align*}
    (1-\varepsilon) \bm{X}^\intercal \bm{U} \bm{X} \prec \bm{X}_s^\intercal \bm{U}_s \bm{X}_s.
\end{align*}
That is, the feasible solution $\Tilde{\bm{u}}_s$ satisfies
\begin{align*}
    (1-\varepsilon) \bm{X}^\intercal \bm{U} \bm{X} 
    \preceq \frac{1-\varepsilon}{\bm{e}^\intercal \bm{u}_s} \bm{X}^\intercal \bm{U}_s \bm{X} 
    \prec \frac{1}{\bm{e}^\intercal \bm{u}_s} \bm{X}_s^\intercal \bm{U}_s \bm{X}_s 
    = \bm{X}_s^\intercal \Tilde{\bm{U}}_s \bm{X}_s.
\end{align*}

We are now ready to prove the final optimality gap.
\begin{theorem}
    Use Algorithm \ref{alg:DLSS_thresh} to construct $\bm{X}_s = \bm{R} \bm{X}$, with $\varepsilon \in (0,1)$. Let $\bm{u}$ be a $\delta$-feasible solution to \eqref{Dual}. Then
    \begin{align*}
        g^* - g_s^*
        < d \log \left( \frac{1 + \delta}{1 - \varepsilon} \right).
    \end{align*}
\end{theorem}
\begin{proof}
    For this proof, we will exploit the fact that
    \begin{align*}
        g^* - g_s^* = \left( g^* - g(\bm{u}) \right) + \left( g(\bm{u}) - g_s^* \right).
    \end{align*}
    We will bound above $g(\bm{u}) - g_s^*$ first. Let $\Tilde{\bm{u}}_s = \frac{1}{\bm{e}^\intercal \bm{u}_s} \bm{u}_s$, where $\bm{u}_s$ contains the first $s$ entries of $\bm{u}$. Since $g_s$ is concave, the optimal value $g_s^*$ must be greater or equal to $g_s$ at any feasible point for \eqref{Duals}. Hence
\begin{align*}
    g_s^* 
    &\geq g_s(\Tilde{\bm{u}}_s) \\
    &= \log \det \left( \bm{X}_s^\intercal \Tilde{\bm{U}}_s \bm{X}_s \right) \\
    &> \log \det \left( (1 - \varepsilon) \bm{X}^\intercal \bm{U} \bm{X} \right) \\
    &= d \log \left( 1 - \varepsilon \right) + \log \det \left( \bm{X}^\intercal \bm{U} \bm{X} \right),
\end{align*}
where the inequality uses Lemma \ref{lem:determinants}.
Hence
\begin{align*}
    g(\bm{u}) - g_s^*
    &< \log \det \left( \bm{X}^\intercal \bm{U} \bm{X} \right) - \left( d \log \left( 1 - \varepsilon \right) + \log \det \left( \bm{X}^\intercal \bm{U} \bm{X} \right) \right) \\
    &= d \log \left( \frac{1}{1 - \varepsilon} \right).
\end{align*}

Now, recall that $\bm{u}$ is $\delta$-feasible for \eqref{Dual}. By Proposition \ref{prop:optimality gap}, we have
\begin{align*}
    g^* - g(\bm{u}) \leq d \log \left( 1 + \delta \right).
\end{align*}
Hence
\begin{align*}
    g^* - g_s^* 
    &= \left( g^* - g(\bm{u}) \right) + \left( g(\bm{u}) - g_s^* \right) \\
    &< d \log \left( \frac{1}{1 - \varepsilon}\right) + d \log \left( 1 + \delta \right) \\
    &= d \log \left( \frac{1 + \delta}{1 - \varepsilon}\right).
\end{align*}
\end{proof}
\section{More numerical results} \label{app:numresults}
\subsection{Comparison of our algorithm with the Fixed Point algorithm} \label{app:comparison}
All computations in this section are performed on a personal laptop with a 64 bit Windows 11 Home operating system (Version 23H2), and a 3.30 GHz  11th Gen Intel Core i7-11370H processor with 40 GB of RAM. The algorithms are run using MATLAB (R2021a).

In this subsection, we compare our algorithm to the Fixed Point algorithm \cite{cohen2019near}. For consistency, we run our algorithm using $s = 10\%$ of $n$. We run both algorithms on Gaussian datasets of size $n = 1\,000\,000$, with dimension $d$ ranging from 2 to 10. Table \ref{tab:morenumericalresults1} summarises the total computation time required for both algorithms, for varying values of accuracy parameter $\delta$. 

\begin{table}[H]
\caption{Time summary for the two algorithms. Tests run on Gaussian datasets of size $n = 1\,000\,000$, and $d$ ranging from 2 to 10. Each value represents mean and sample standard deviation of 100 runs.}
\centering
\vspace{11pt}
\begin{tabular}{lllll}
& & & \multicolumn{2}{c}{Time (s)} \\
\cmidrule(r){4-5}
$n$ & $d$ & $\delta$ & Our algorithm & Fixed Point algorithm \\ 
\midrule
$1\,000\,000$ & 2 & $10^{-2}$ & 0.0914 $\pm$ 0.0220 & 1.4331 $\pm$ 0.0877 \\
& & $10^{-3}$ & 0.0922 $\pm$ 0.0210 & 12.1627 $\pm$ 0.8764 \\
& & $10^{-4}$ & 0.0873 $\pm$ 0.0215 & 63.1388 $\pm$ 6.2418 \\
& 5 & $10^{-2}$ & 0.1330 $\pm$ 0.0219 & 3.6558 $\pm$ 0.2074 \\
& & $10^{-3}$ & 0.1956 $\pm$ 0.0442 & 12.1783 $\pm$ 1.3380 \\
& & $10^{-4}$ & 0.2125 $\pm$ 0.0390 & 47.2873 $\pm$ 11.6625 \\
& 10 & $10^{-2}$ & 0.2483 $\pm$ 0.0618 & 6.0298 $\pm$ 0.7460 \\
& & $10^{-3}$ & 0.4177 $\pm$ 0.1025 & 21.8280 $\pm$ 3.9056 \\
& & $10^{-4}$ & 0.4313 $\pm$ 0.1077 & 117.9867 $\pm$ 41.5669
\end{tabular}
\label{tab:morenumericalresults1}
\end{table}

We now apply the same methods to Gaussian matrices of dimension $d = 100$, with varying number of points $n$ and accuracy parameter $\delta$. Table \ref{tab:morenumericalresults2} summarises the total computation time required for both algorithms.

\begin{table}[H]
\caption{Time summary for the two algorithms. Tests run on Gaussian datasets of dimension $d = 100$, with varying $n$. Each value represents mean and sample standard deviation (when multiple runs are computed).}
\centering
\vspace{11pt}
\begin{tabular}{llllll}
& & & & \multicolumn{2}{c}{Time (s)} \\
\cmidrule(r){5-6}
$d$ & $n$ & $\delta$ & $\#$ runs & Our algorithm & Fixed Point algorithm\\ 
\midrule
100 & $10\,000$ & $10^{-2}$ & 100 & 0.0927 $\pm$ 0.0402 & 0.7270 $\pm$ 0.1886 \\
& & $10^{-3}$ & 100 & 0.2578 $\pm$ 0.1991 & 5.6216 $\pm$ 1.5382 \\
& & $10^{-4}$ & 100 & 0.3013 $\pm$ 0.0833 & 18.6236 $\pm$ 4.8307 \\
& & $10^{-5}$ & 50 & 0.5162 $\pm$ 0.1292 & 82.4097 $\pm$ 26.7802 \\
& $100\,000$ & $10^{-3}$ & 20 & 1.9359 $\pm$ 0.1371 & 160.1398 $\pm$ 34.8062 \\
& $1\,000\,000$ & $10^{-3}$ & 1 & 101.0625 & 4837.5313
\end{tabular}
\label{tab:morenumericalresults2}
\end{table}

In both Tables \ref{tab:morenumericalresults1} and \ref{tab:morenumericalresults2}, we notice that as $\delta$ increases, the time taken for both methods generally increases. However, the increase in time for our algorithm is minimal, while the increase in time for the Fixed Point algorithm is significant. In the final row of Table \ref{tab:morenumericalresults2}, our algorithm takes less than 2 minutes, compared with the Fixed Point algorithm, which took about 80 minutes. 

We emphasise that the examples considered here are smaller than those we considered in Section \ref{sec:empirical_results}, and the accuracy parameter $\delta$ is much larger than our desired $10^{-9}$. Regardless, we ran the Fixed Point algorithm on the examples from Section \ref{sec:empirical_results}. We found that for all three datasets, the algorithm did not converge in two hours of runtime.
\subsection{Real world data} \label{sec:realworlddata}
All computations in this section are performed on a personal laptop with a 64 bit MacOS 13 operating system, and a 2.4 GHz Quad-Core Intel Core i5 processor with 8 GB of RAM. The algorithms are run using MATLAB (R2021a).

We now test our algorithm on three smaller datasets from the UCI Machine Learning Repository. Table \ref{tab:datasetoutline1} summarises the size of each dataset, and time to compute the MVCE using the WA algorithm \cite{wolfe1970convergence,atwood1973sequences}.

\begin{table}[h]
\caption{Dataset description. This includes dimension and time to compute MVCE using the WA algorithm.}
\centering
\vspace{11pt}
\begin{tabular}{llll}
Dataset & $n$ & $d$ & Time (s) \\
\midrule
Ethylene CO \cite{gas_sensor2015} & $4\,208\,261$ & 19 & 85.92 $\pm$ 0.17 \\
Ethylene CH4 \cite{gas_sensor2015} & $4\,178\,504$ & 19 & 75.89 $\pm$ 0.38 \\
Skin \cite{skin2012} & $245\,057$ & 4 & 0.54 $\pm$ 0.04
\end{tabular}
\label{tab:datasetoutline1}
\end{table}

We then apply our algorithm on each dataset, varying sample size $s$ from 1$\%$ to 10$\%$ of $n$. In Table \ref{tab:optgap1}, we summarise the optimality gaps. The deterministic sampling performed well on all datasets, achieving very small optimality gaps when $s$ is 10$\%$ of $n$. In Table \ref{tab:time1}, we summarise the computation times. Although the deterministic sampling was the slowest, all sampling methods decreased the total computation time required to calculate the MVCE.

\begin{table}[t]
\caption{Optimality gap summary. Sample size $s$ takes values of 1$\%$, 5$\%$, and 10$\%$ of $n$. Each value represents mean and sample standard deviation of 5 runs.}
\centering
\vspace{11pt}
\begin{tabular}{lllll}
& & \multicolumn{3}{c}{Optimality Gap} \\
\cmidrule(r){3-5}
Dataset & $s$ & det & prob & unif \\
\midrule
Ethylene CO \cite{gas_sensor2015} & 1$\%$ & 4.66 & 1.01 $\pm$ 0.11 & 5.47 $\pm$ 0.65 \\
& 5$\%$ & 0.02 & 0.43 $\pm$ 0.03 & 4.40 $\pm$ 1.03 \\
& 10$\%$ & 5.67E-06 & 0.25 $\pm$ 0.04 & 3.24 $\pm$ 0.48 \\
Ethylene CH4 \cite{gas_sensor2015} & 1$\%$ & 1.59 & 1.48 $\pm$ 0.09 & 3.53 $\pm$ 0.42 \\
& 5$\%$ & 0.05 & 0.66 $\pm$ 0.05 & 1.69 $\pm$ 0.15 \\
& 10$\%$ & 0.01 & 0.37 $\pm$ 0.08 & 1.35 $\pm$ 0.11 \\
Skin \cite{skin2012} & 1$\%$ & 0.75 & 0.07 $\pm$ 0.02 & 0.25 $\pm$ 0.13 \\
& 5$\%$ & 0.56 & 0.02 $\pm$ 0.01 & 0.05 $\pm$ 0.03 \\
& 10$\%$ & -3.55E-15 & 0.03 $\pm$ 0.01 & 0.04 $\pm$ 0.02
\end{tabular}
\label{tab:optgap1}
\end{table}

\begin{table}[t]
\caption{Time summary. Sample size $s$ takes values of 1$\%$, 5$\%$, and 10$\%$ of $n$. Each value represents mean and sample standard deviation of 5 runs.}
\centering
\vspace{11pt}
\begin{tabular}{lllll}
& & \multicolumn{3}{c}{Time (s)} \\
\cmidrule(r){3-5}
Dataset & $s$ & det & prob & unif \\
\midrule
Ethylene CO \cite{gas_sensor2015} & 1$\%$ & 6.00 $\pm$ 0.12 & 5.42 $\pm$ 0.74 & 4.01 $\pm$ 0.39 \\
& 5$\%$ & 15.00 $\pm$ 0.56 & 9.94 $\pm$ 0.47 & 7.55 $\pm$ 0.05 \\
& 10$\%$ & 15.71 $\pm$ 0.14 & 14.32 $\pm$ 1.26 & 11.66 $\pm$ 0.95 \\
Ethylene CH4 \cite{gas_sensor2015} & 1$\%$ & 5.91 $\pm$ 0.30 & 5.24 $\pm$ 0.36 & 3.35 $\pm$ 0.30 \\
& 5$\%$ & 13.56 $\pm$ 0.58 & 10.00 $\pm$ 1.10 & 7.95 $\pm$ 1.08 \\
& 10$\%$ & 15.18 $\pm$ 0.34 & 13.05 $\pm$ 0.81 & 10.26 $\pm$ 0.93 \\
Skin \cite{skin2012} & 1$\%$ & 0.04 $\pm$ 0.01 & 0.04 $\pm$ 0.01 & 0.05 $\pm$ 0.03 \\
& 5$\%$ & 0.15 $\pm$ 0.03 & 0.13 $\pm$ 0.01 & 0.10 $\pm$ 0.02 \\
& 10$\%$ & 0.15 $\pm$ 0.03 & 0.17 $\pm$ 0.03 & 0.13 $\pm$ 0.02
\end{tabular}
\label{tab:time1}
\end{table} 

\end{document}